\documentclass[10pt,reqno,a4paper,oneside,11pt]{amsart}%
\usepackage{tikz}

\usepackage{mathtools}
\usepackage{extarrows}
\usetikzlibrary{arrows}

\usepackage{xparse}
\usepackage{amsfonts}
\usepackage{amsfonts}
\usepackage{amsfonts}
\usepackage{amsfonts}
\usepackage{amsfonts}
\usepackage{amsfonts}
\usepackage{amsfonts}
\usepackage{amsfonts}
\usepackage{amsfonts}
\usepackage{amsfonts}
\usepackage{amsfonts}
\usepackage{amsfonts}
\usepackage{amsfonts}
\usepackage{amsfonts}
\usepackage{amsfonts}
\usepackage{amsfonts}
\usepackage{amsfonts}
\usepackage{amsfonts}
\usepackage{amsfonts}
\usepackage{amsfonts}
\usepackage{mathrsfs}
\usepackage{mathrsfs}
\usepackage{amsfonts}
\usepackage{amssymb}
\usepackage{amsmath}
\usepackage{amsthm}
\usepackage{graphicx}%
\providecommand{\U}[1]{\protect\rule{.1in}{.1in}}
\newtheorem{theorem}{Theorem}[section]

\newtheorem{lemma}[theorem]{Lemma}

\newtheorem{definition}{Definition}[section]

\theoremstyle{definition}
\theoremstyle{remark}
\numberwithin{equation}{section}

\ifx\pdfoutput\relax\let\pdfoutput=\undefined\fi
\newcount\msipdfoutput
\ifx\pdfoutput\undefined\else
\ifcase\pdfoutput\else
\ifx\paperwidth\undefined\else
\ifdim\paperheight=0pt\relax\else\pdfpageheight\paperheight\fi
\ifdim\paperwidth=0pt\relax\else\pdfpagewidth\paperwidth\fi
\fi\fi\fi
\begin{document}
\pagestyle{myheadings}

\begin{center}
{\huge \textbf{Sylvester-type quaternion matrix equations with arbitrary equations and arbitrary unknowns}}\footnote{Email address: zhuohenghe@shu.edu.cn, hzh19871126@126.com}

\bigskip

{ \textbf{Zhuo-Heng He}}

{\small
\vspace{0.25cm}

Department of Mathematics, Shanghai University, Shanghai 200444, P. R. China}

\end{center}

\begin{quotation}
\noindent\textbf{Abstract:}
In this paper, we prove a conjecture which was presented in a recent paper [Linear
Algebra Appl. 2016; 496: 549--593]. We derive some practical necessary and sufficient conditions for the existence of a solution to a system of coupled two-sided Sylvester-type quaternion matrix equations with arbitrary equations and arbitrary unknowns $A_{i}X_{i}B_{i}+C_{i}X_{i+1}D_{i}=E_{i},~i=\overline{1,k}$.
As an application, we give some practical necessary and sufficient conditions for the existence of an $\eta$-Hermitian solution to the system of quaternion matrix equations $A_{i}X_{i}A^{\eta*}_{i}+C_{i}X_{i+1}C^{\eta*}_{i}=E_{i}$ in terms of ranks, $~i=\overline{1,k}$.
\newline%
\end{quotation}

\section{\textbf{Introduction}}

In this paper, we consider the the system of coupled two-sided Sylvester-type quaternion matrix equations with $k$ equations and $k+1$ unknowns
\begin{align}\label{march15equ031}
\begin{cases}
&A_{1}X_{1}B_{1}+C_{1}X_{2}D_{1}=E_{1},\\
&\qquad \qquad \quad A_{2}X_{2}B_{2}+C_{2}X_{3}D_{2}=E_{2},\\
&\qquad \qquad \quad \qquad \qquad \quad A_{3}X_{3}B_{3}+C_{3}X_{4}D_{3}=E_{3},\\
&\qquad \qquad \quad \qquad \qquad\qquad\qquad\qquad\ddots\\
&\qquad \qquad \quad \qquad \qquad\qquad\qquad\qquad A_{k}X_{k}B_{k}+C_{k}X_{k+1}D_{k}=E_{k},
\end{cases}
\end{align}
where $A_{i},B_{i}, C_{i},D_{i},$ and $E_{i}$ are given matrices, $X_{i}$ are unknowns. Since Baksalary \cite{JKB2} first studied the system (\ref{march15equ031}) for the case $k=1$ over the field in 1980, there have been many papers to consider the case $k=1$ (e.g., \cite{A.B}, \cite{wqw01}). In 2016, He \emph{et.al} \cite{helaa} investigated a simultaneous decomposition to consider the system (\ref{march15equ031}) for the case $k=3.$ They established some necessary and sufficient conditions for the existence of a solution to the system (\ref{march15equ031}) in terms of ranks of the matrices involved \cite{helaa}. At the end of the paper \cite{helaa}, they gave a conjecture on the solvability condition to the system (\ref{march15equ031}) in terms of ranks for the case $k\geq4,$ see Theorem \ref{maintheorem}. Notice however, it is hard to solve the conjecture using the approach presented in \cite{helaa}.

In this paper, we use another approach to prove Theorem \ref{maintheorem}, i.e., the conjecture which proposed in \cite{helaa}.  We then consider solvability conditions to the system of quaternion matrix equations involving $\eta$-Hermicity
\begin{align}\label{applicationsystem}
\begin{cases}
&A_{1}X_{1}A^{\eta*}_{1}+C_{1}X_{2}C^{\eta*}_{1}=E_{1},\\
&\qquad \qquad \quad A_{2}X_{2}A^{\eta*}_{2}+C_{2}X_{3}C^{\eta*}_{2}=E_{2},\\
&\qquad \qquad \quad \qquad \qquad \quad A_{3}X_{3}A^{\eta*}_{3}+C_{3}X_{4}C^{\eta*}_{3}=E_{3},\\
&\qquad \qquad \quad \qquad \qquad\qquad\qquad\qquad\ddots\\
&\qquad \qquad \quad \qquad \qquad\qquad\qquad\qquad A_{k}X_{k}A^{\eta*}_{k}+C_{k}X_{k+1}C^{\eta*}_{k}=E_{k},
\end{cases}~X_{i}=X^{\eta*}_{i}.
\end{align}

The remainder of this paper is organized as follows. In Section 2, we give the main result of this paper. We derive some practical necessary and sufficient conditions for the existence of a solution to the system (\ref{march15equ031}). In Section 3, we give the proof of Theorem \ref{maintheorem}. In Section 4, we derive some practical necessary and sufficient conditions for the existence of an $\eta$-Hermitian solution to the system (\ref{applicationsystem}) in terms of ranks, see Theorem \ref{theorem061}.

Let $\mathbb{R}$ and $\mathbb{H}^{m\times n}$ stand, respectively, for the real field and the space of all $m\times n$ matrices over the real quaternion algebra
\[
\mathbb{H}%
=\big\{a_{0}+a_{1}\mathbf{i}+a_{2}\mathbf{j}+a_{3}\mathbf{k}\big|~\mathbf{i}^{2}=\mathbf{j}^{2}=\mathbf{k}^{2}=\mathbf{ijk}=-1,a_{0}%
,a_{1},a_{2},a_{3}\in\mathbb{R}\big\}.
\]
 The symbols $r(A)$ and $A^{\ast}$ stand for the rank of a given quaternion matrix $A$ and the conjugate transpose of $A$ and the transposed of $A$, respectively. $I$ and $0$ are the identity matrix and zero matrix with appropriate sizes, respectively. The Moore-Penrose inverse
$A^{\dag}$ of a quaternion matrix $A$, is defined to be the unique matrix $A^{\dag},$
such that
\begin{align*}
\text{(i)}~AA^{\dag}A=A,~\text{(ii)}~A^{\dag}AA^{\dag}=A^{\dag},~\text{(iii)}%
~(AA^{\dag})^{\ast}=AA^{\dag},~\text{(iv)}~(A^{\dag}A)^{\ast}=A^{\dag}A.
\end{align*}Furthermore, $L_{A}$ and $R_{A}$ stand for the projectors $L_{A}%
=I-A^{\dag}A$ and $R_{A}=I-AA^{\dag}$ induced by $A$, respectively. For more definitions and properties of quaternions, we refer the reader to the book \cite{rodman}.

\section{\textbf{The main result}}
 In this section, we give some practical necessary and sufficient conditions for the existence of a solution to the system (\ref{march15equ031}) in terms of ranks.
\begin{theorem}\label{maintheorem}
The system (\ref{march15equ031}) is consistent if and only if the following $2k(k+1)$ rank equalities hold for all $i=1,\ldots,k$ and $1\leq m< n\leq k$
\begin{align}\label{march15equ032}
r\begin{pmatrix}
A_{i}&E_{i}&C_{i}
\end{pmatrix}=r\begin{pmatrix}
A_{i}&C_{i}
\end{pmatrix},~
r\begin{pmatrix}
B_{i}\\E_{i}\\D_{i}
\end{pmatrix}=r\begin{pmatrix}B_{i}\\D_{i}\end{pmatrix},
\end{align}
\begin{align}\label{march15equ033}
r\begin{pmatrix}
A_{i}&E_{i}\\
0&D_{i}
\end{pmatrix}=r(A_{i})+r(D_{i}),~
r\begin{pmatrix}
B_{i}&0\\E_{i}&C_{i}
\end{pmatrix}=r(B_{i})+r(C_{i}),
\end{align}
\begin{align}\label{march15equ034}
&r\begin{pmatrix}\begin{smallmatrix}
A_{m}&E_{m}&C_{m}\\
 &D_{m}& &B_{m+1}\\
 & & A_{m+1}&-E_{m+1}&C_{m+1}\\
  & & & \ddots & \ddots & \ddots\\
  & & & & A_{n}&(-1)^{n-m}E_{n}&C_{n}\end{smallmatrix}
\end{pmatrix}\nonumber\\&=r
\begin{pmatrix}\begin{smallmatrix}
A_{m}&C_{m}\\
& A_{m+1}&C_{m+1}\\
& & \ddots&\ddots\\
& & & A_{n}&C_{n}\end{smallmatrix}
\end{pmatrix}+r
\begin{pmatrix}\begin{smallmatrix}
D_{m}&B_{m+1}\\
& D_{m+1}&B_{m+2}\\
& & \ddots&\ddots\\
& & & D_{n-1}&B_{n}\end{smallmatrix}
\end{pmatrix},
\end{align}
\begin{align}\label{march15equ035}
&r
\begin{pmatrix}\begin{smallmatrix}
B_{m}\\
E_{m}&C_{m}\\
D_{m}& &B_{m+1}\\
 &A_{m+1}&-E_{m+1}&\ddots\\
 & & D_{m+1}& \ddots & B_{n}\\
  & & &\ddots&(-1)^{n-m}E_{n}\\
  & & & & D_{n}\end{smallmatrix}
\end{pmatrix}\nonumber\\&=r
\begin{pmatrix}\begin{smallmatrix}
C_{m}\\
 A_{m+1}&C_{m+1}\\
 &A_{m+2}&\ddots\\
 & &\ddots&C_{n-1}\\
 & & &A_{n}\end{smallmatrix}
\end{pmatrix}+r
\begin{pmatrix}\begin{smallmatrix}
B_{m}\\
D_{m}&B_{m+1}\\
& D_{m+1}&\ddots\\
& & \ddots&B_{n}\\
& & & D_{n}\end{smallmatrix}
\end{pmatrix},
\end{align}
\begin{align}\label{march15equ036}
&r
\begin{pmatrix}\begin{smallmatrix}
A_{m}&E_{m}&C_{m}\\
&D_{m}& &B_{m+1}\\
& &A_{m+1}&-E_{m+1}&\ddots\\
& & & D_{m+1}& \ddots & B_{n}\\
&  & & &\ddots&(-1)^{n-m}E_{n}\\
&  & & & & D_{n}\end{smallmatrix}
\end{pmatrix}\nonumber\\&=r
\begin{pmatrix}\begin{smallmatrix}
A_{m}&C_{m}\\
& A_{m+1}&C_{m+1}\\
& & \ddots&\ddots\\
& & & A_{n-1}&C_{n-1}\\
& & & & A_{n}\end{smallmatrix}
\end{pmatrix}+r
\begin{pmatrix}\begin{smallmatrix}
D_{m}&B_{m+1}\\
& D_{m+1}&B_{m+2}\\
& & \ddots&\ddots\\
& & & D_{n-1}&B_{n}\\
& & & & D_{n}\end{smallmatrix}
\end{pmatrix},
\end{align}
\begin{align}\label{march15equ037}
&r
\begin{pmatrix}\begin{smallmatrix}
B_{m}\\
E_{m}&C_{m}\\
D_{m}& &B_{m+1}\\
& A_{m+1}&-E_{m+1}&C_{m+1}\\
 & & \ddots & \ddots & \ddots\\
 & & & A_{n}&(-1)^{n-m}E_{n}&C_{n}\end{smallmatrix}
\end{pmatrix}\nonumber\\&=r
\begin{pmatrix}\begin{smallmatrix}
C_{m}\\
A_{m+1}&C_{m+1}\\
&A_{m+2}&C_{m+2}\\
&&\ddots &\ddots\\
& & & A_{n}&C_{n}\end{smallmatrix}
\end{pmatrix}+r
\begin{pmatrix}\begin{smallmatrix}
B_{m}\\
D_{m+1}&B_{m+1}\\
&D_{m+2}&B_{m+2}\\
&&\ddots &\ddots\\
& & & D_{n-1}&B_{n}\end{smallmatrix}
\end{pmatrix},
\end{align}where the blank entries in (\ref{march15equ034})-(\ref{march15equ037}) are all zeros.
\end{theorem}

\section{\textbf{Proof of Theorem \ref{maintheorem}}}

In this section, we give the proof of Theorem \ref{maintheorem}. We proceed with the proof by induction.  Lemma \ref{lemma01} proves Theorem \ref{maintheorem} for the case $k=1$.

\begin{lemma}\label{lemma01}(\cite{JKB2}, \cite{A.B}, \cite{wqw01})
Consider the quaternion matrix equation
\begin{align}\label{march16equ041}
A_{1}X_{1}B_{1}+C_{1}X_{2}D_{1}=E_{1}.
\end{align} Let $M_{1}=R_{A_{1}}C_{1}$, $N_{1}=D_{1}L_{B_{1}}$,
$S_{1}=C_{1}L_{M_{1}}$. Then the following statements are
equivalent:\newline(1) Equation (\ref{march16equ041}) is consistent.\newline%
(2)
\begin{align}
R_{M_{1}}R_{A_{1}}E=0,~ EL_{B}L_{N_{1}}=0,~ R_{A_{1}}E_{1}L_{D_{1}}=0,~ R_{C_{1}}E_{1}L_{B_{1}}=0.
\end{align}
(3)
\begin{align}\label{march16equ043}
r\begin{pmatrix}
A_{1}&E_{1}&C_{1}
\end{pmatrix}=r\begin{pmatrix}
A_{1}&C_{1}
\end{pmatrix},~
r\begin{pmatrix}
B_{1}\\E_{1}\\D_{1}
\end{pmatrix}=r\begin{pmatrix}B_{1}\\D_{1}\end{pmatrix},
\end{align}
\begin{align} \label{march16equ044}
r\begin{pmatrix}
A_{1}&E_{1}\\
0&D_{1}
\end{pmatrix}=r(A_{1})+r(D_{1}),~
r\begin{pmatrix}
B_{1}&0\\E_{1}&C_{1}
\end{pmatrix}=r(B_{1})+r(C_{1}).
\end{align}
In this case, the general solution to (\ref{march16equ041}) can be
expressed as
\begin{align}
X_{1}=A_{1}^{\dag}E_{1}B_{1}^{\dag}-A_{1}^{\dag}C_{1}M_{1}^{\dag}E_{1}B_{1}^{\dag}-A_{1}^{\dag}S_{1}%
C_{1}^{\dag}E_{1}N_{1}^{\dag}D_{1}B_{1}^{\dag}-A_{1}^{\dag}S_{1}Y_{1}R_{N_{1}}D_{1}B_{1}^{\dag}+L_{A_{1}}%
Y_{2}+Y_{3}R_{B_{1}},
\end{align}
\begin{align}
X_{2}=M_{1}^{\dag}E_{1}D_{1}^{\dag}+S_{1}^{\dag}S_{1}C_{1}^{\dag}E_{1}N_{1}^{\dag}%
+L_{M_{1}}L_{S_{1}}Y_{4}+Y_{5}R_{D_{1}}+L_{M_{1}}Y_{1}R_{N_{1}},
\end{align}
where $Y_{1},Y_{2},Y_{3},Y_{4},Y_{5}$ are arbitrary matrices over
$\mathbb{H}$ with appropriate sizes. As a special case of (\ref{march16equ041}), the matrix equation
\begin{align}\label{march17equ042}
A_{1}X_{1}+X_{2}D_{1}=E_{1}
\end{align}is consistent if and only if
\begin{align}\label{march17equ048}
R_{A_{1}}E_{1}L_{D_{1}}=0.
\end{align}
\end{lemma}

To simplify the ranks of the proof of Theorem \ref{maintheorem}, we need the following lemma.

\begin{lemma}
\label{hlemma03}\cite{GPH}. Let $A\in\mathbb{H}^{m\times n},B\in
\mathbb{H}^{m\times k},$ and $C\in\mathbb{H}^{l\times
n}$ be given. Then\newline
$(1)~ r(A)+r(R_{A}B)=r(B)+r(R_{B}A)=r\begin{pmatrix}A&B\end{pmatrix}.$\newline
$ (2)~ r(A)+r(CL_{A})=r(C)+r(AL_{C})=r%
\begin{pmatrix}
A\\
C
\end{pmatrix}
.$
\end{lemma}

Now we give the proof of Theorem \ref{maintheorem}.

\textbf{Proof of Theorem \ref{maintheorem}:} We proceed with the proof by induction. If $k=1$, then the system (\ref{march15equ031}) becomes the equation (\ref{march16equ041}) and the rank equalities (\ref{march15equ032})-(\ref{march15equ037}) become (\ref{march16equ043})-(\ref{march16equ044}). Hence, the statement is true if $k=1.$

Suppose the statement is true when the number of the equations is $k-1$. We show by induction that it is
true for the number of the equations $k$.

We separate the system (\ref{march15equ031}) into $k$ parts
\begin{align}
&A_{1}X_{1}B_{1}+C_{1}X_{2}D_{1}=E_{1},\label{march16equ047}
\\
&A_{2}X_{2}B_{2}+C_{2}X_{3}D_{2}=E_{2},\label{march16equ048}
\\
&A_{3}X_{3}B_{3}+C_{3}X_{4}D_{3}=E_{3},\label{march16equ049}
\\
&\qquad \qquad \vdots \nonumber
\\
&A_{k}X_{k}B_{k}+C_{k}X_{k+1}D_{k}=E_{k}.\label{march16equ0410}
\end{align}
It following from Lemma \ref{lemma01} that the equations (\ref{march16equ047})-(\ref{march16equ0410}) are consistent if and only if the rank equalities (\ref{march15equ032}) and (\ref{march15equ033}) hold. The general solution to the equation
\begin{align}
&A_{i}X_{i}B_{i}+C_{i}X_{i+1}D_{i}=E_{i}\label{march16equ0411}
\end{align}can be expressed as
\begin{align}
X_{i}=A_{i}^{\dag}E_{i}B_{i}^{\dag}-A_{i}^{\dag}C_{i}M_{i}^{\dag}E_{i}B_{i}^{\dag}-A_{i}^{\dag}S_{i}%
C_{i}^{\dag}E_{i}N_{i}^{\dag}D_{i}B_{i}^{\dag}-A_{i}^{\dag}S_{i}Y_{i}R_{N_{i}}D_{i}B_{i}^{\dag}+L_{A_{i}}%
Z^{(i)}_{1}+Z^{(i)}_{2}R_{B_{i}},
\end{align}
\begin{align}
X_{i+1}=M_{i}^{\dag}E_{i}D_{i}^{\dag}+S_{i}^{\dag}S_{i}C_{i}^{\dag}E_{i}N_{i}^{\dag}%
+L_{M_{i}}L_{S_{i}}Z^{(i)}_{3}+Z^{(i)}_{4}R_{D_{i}}+L_{M_{i}}Y_{i}R_{N_{i}},
\end{align}
where
\begin{align}
M_{i}=R_{A_{i}}C_{i}, ~N_{i}=D_{i}L_{B_{i}},~S_{i}=C_{i}L_{M_{i}},
\end{align}
and $Y_{i},Z^{(i)}_{1},Z^{(i)}_{2},Z^{(i)}_{3},Z^{(i)}_{4}$ are arbitrary matrices over
$\mathbb{H}$ with appropriate sizes.

Let $X_{i+1}$ in the $i$th equation be equal to $X_{i+1}$ in the $(i+1)$th equation for every $i=1,\ldots,k-1.$ Then, we have the following system
\scriptsize
\begin{align}\label{march17equ417}
\begin{cases}
\begin{pmatrix}L_{M_{1}}L_{S_{1}}&-L_{A_{2}}\end{pmatrix}\begin{pmatrix}Z^{(1)}_{3}\\Z^{(2)}_{1}\end{pmatrix}
+\begin{pmatrix}Z^{(1)}_{4}&Z^{(2)}_{2}\end{pmatrix}\begin{pmatrix}R_{D_{1}}\\-R_{B_{2}}\end{pmatrix}
=F_{1}-L_{M_{1}}Y_{1}R_{N_{1}}-A_{2}^{\dag}S_{2}Y_{2}R_{N_{2}}D_{2}B_{2}^{\dag},\\
\begin{pmatrix}L_{M_{2}}L_{S_{2}}&-L_{A_{3}}\end{pmatrix}\begin{pmatrix}Z^{(2)}_{3}\\Z^{(3)}_{1}\end{pmatrix}
+\begin{pmatrix}Z^{(2)}_{4}&Z^{(3)}_{2}\end{pmatrix}\begin{pmatrix}R_{D_{2}}\\-R_{B_{3}}\end{pmatrix}
=F_{2}-L_{M_{2}}Y_{2}R_{N_{2}}-A_{3}^{\dag}S_{3}Y_{3}R_{N_{3}}D_{3}B_{3}^{\dag},\\
\vdots\\
\begin{pmatrix}L_{M_{j}}L_{S_{j}}&-L_{A_{j+1}}\end{pmatrix}\begin{pmatrix}Z^{(j)}_{3}\\Z^{(j+1)}_{1}\end{pmatrix}
+\begin{pmatrix}Z^{(j)}_{4}&Z^{(j+1)}_{2}\end{pmatrix}\begin{pmatrix}R_{D_{j}}\\-R_{B_{j+1}}\end{pmatrix}
=F_{j}-L_{M_{j}}Y_{j}R_{N_{j}}-A_{j+1}^{\dag}S_{j+1}Y_{j+1}R_{N_{j+1}}D_{j+1}B_{j+1}^{\dag},\\
\vdots\\
\begin{pmatrix}L_{M_{k-1}}L_{S_{k-1}}&-L_{A_{k}}\end{pmatrix}\begin{pmatrix}Z^{(k-1)}_{3}\\Z^{(k)}_{1}\end{pmatrix}
+\begin{pmatrix}Z^{(k-1)}_{4}&Z^{(k)}_{2}\end{pmatrix}\begin{pmatrix}R_{D_{k-1}}\\-R_{B_{k}}\end{pmatrix}
=F_{k-1}-L_{M_{k-1}}Y_{k-1}R_{N_{k-1}}-A_{k}^{\dag}S_{k}Y_{k}R_{N_{k}}D_{k}B_{k}^{\dag},
\end{cases}
\end{align}\normalsize
where $j=\overline{1,k-1},$ and
\begin{align}\label{march17equfj}
F_{j}=&
A_{j+1}^{\dag}E_{j+1}B_{j+1}^{\dag}-A_{j+1}^{\dag}C_{j+1}M_{j+1}^{\dag}E_{j+1}B_{j+1}^{\dag}-A_{j+1}^{\dag}S_{j+1}%
C_{j+1}^{\dag}E_{j+1}N_{j+1}^{\dag}D_{j+1}B_{j+1}^{\dag}\nonumber\\&
-(M_{j}^{\dag}E_{j}D_{j}^{\dag}+S_{j}^{\dag}S_{j}C_{j}^{\dag}E_{j}N_{j}^{\dag}).
\end{align}Hence, the system (\ref{march15equ031}) is consistent if and only if (\ref{march15equ032}) and (\ref{march15equ033}) hold and the system (\ref{march17equ417}) is consistent. We now turn our attention to the solvability conditions to the system (\ref{march17equ417}). Observe that each equation in  the system (\ref{march17equ417}) has the form of (\ref{march17equ042}). It follows from the condition (\ref{march17equ048}) in Lemma \ref{lemma01} that the equation
\begin{align}
&\begin{pmatrix}L_{M_{j}}L_{S_{j}}&-L_{A_{j+1}}\end{pmatrix}\begin{pmatrix}Z^{(j)}_{3}\\Z^{(j+1)}_{1}\end{pmatrix}
+\begin{pmatrix}Z^{(j)}_{4}&Z^{(j+1)}_{2}\end{pmatrix}\begin{pmatrix}R_{D_{j}}\\-R_{B_{j+1}}\end{pmatrix}
\nonumber\\&=F_{j}-L_{M_{j}}Y_{j}R_{N_{j}}-A_{j+1}^{\dag}S_{j+1}Y_{j+1}R_{N_{j+1}}D_{j+1}B_{j+1}^{\dag}
\end{align}is consistent if and only if
\begin{align}\label{march17equ420}
R_{\begin{pmatrix}\begin{smallmatrix}L_{M_{j}}L_{S_{j}}&-L_{A_{j+1}}\end{smallmatrix}\end{pmatrix}}
\left(F_{j}-L_{M_{j}}Y_{j}R_{N_{j}}-A_{j+1}^{\dag}S_{j+1}Y_{j+1}R_{N_{j+1}}D_{j+1}B_{j+1}^{\dag}\right)
L_{\begin{pmatrix}\begin{smallmatrix}R_{D_{j}}\\-R_{B_{j+1}}\end{smallmatrix}\end{pmatrix}}=0.
\end{align}Put
\begin{align}\label{may28equ421}
\widehat{A_{j}}=R_{\begin{pmatrix}\begin{smallmatrix}L_{M_{j}}L_{S_{j}}&-L_{A_{j+1}}\end{smallmatrix}\end{pmatrix}}L_{M_{j}},~
\widehat{B_{j}}=R_{N_{j}}L_{\begin{pmatrix}\begin{smallmatrix}R_{D_{j}}\\-R_{B_{j+1}}\end{smallmatrix}\end{pmatrix}},
\end{align}
\begin{align}\label{may28equ422}
\widehat{C_{j}}=R_{\begin{pmatrix}\begin{smallmatrix}L_{M_{j}}L_{S_{j}}&-L_{A_{j+1}}\end{smallmatrix}\end{pmatrix}}A_{j+1}^{\dag}S_{j+1},~
\widehat{D_{j}}=R_{N_{j+1}}D_{j+1}B_{j+1}^{\dag}L_{\begin{pmatrix}\begin{smallmatrix}R_{D_{j}}\\-R_{B_{j+1}}\end{smallmatrix}\end{pmatrix}},
\end{align}
\begin{align}\label{may28equ423}
\widehat{E_{j}}=R_{\begin{pmatrix}\begin{smallmatrix}L_{M_{j}}L_{S_{j}}&-L_{A_{j+1}}\end{smallmatrix}\end{pmatrix}}F_{j}
L_{\begin{pmatrix}\begin{smallmatrix}R_{D_{j}}\\-R_{B_{j+1}}\end{smallmatrix}\end{pmatrix}}.
\end{align}Then the equation (\ref{march17equ420}) becomes the following form
\begin{align}
\widehat{A_{j}}Y_{j}\widehat{B_{j}}+\widehat{C_{j}}Y_{j+1}\widehat{D_{j}}=\widehat{E_{j}}.
\end{align}Thus, the system (\ref{march17equ417}) is consistent if and only if the following system is consistent
\begin{align}\label{march17equ424}
\begin{cases}
&\widehat{A_{1}}Y_{1}\widehat{B_{1}}+\widehat{C_{1}}Y_{2}\widehat{D_{1}}=\widehat{E_{1}},\\
&\qquad \qquad \quad \widehat{A_{2}}Y_{2}\widehat{B_{2}}+\widehat{C_{2}}Y_{3}\widehat{D_{2}}=\widehat{E_{2}},\\
&\qquad \qquad \quad \qquad \qquad \ddots\\
&\qquad \qquad \quad \qquad \qquad \widehat{A_{k-1}}Y_{k-1}\widehat{B_{k-1}}+\widehat{C_{k-1}}Y_{k}\widehat{D_{k-1}}=\widehat{E_{k-1}},
\end{cases}
\end{align}where $\widehat{A_{j}},~\widehat{B_{j}},~\widehat{C_{j}},~\widehat{D_{j}},~\widehat{E_{j}}$ are defined in (\ref{may28equ421})-(\ref{may28equ423}).
Note that the form of the system (\ref{march17equ424}) is same as the main system (\ref{march15equ031}) and the number of the equations in (\ref{march17equ424}) is $k-1$. Applying the induction hypothesis on the system (\ref{march17equ424}), we have that the system (\ref{march17equ424}) is consistent if and only if the following $2(k-1)k$ rank equalities hold for all $j=1,\ldots,k-1$ and $1\leq m< n\leq k-1$
\begin{align}\label{march17equ425}
r\begin{pmatrix}
\widehat{A_{j}}&\widehat{E_{j}}&\widehat{C_{j}}
\end{pmatrix}=r\begin{pmatrix}
\widehat{A_{j}}&\widehat{C_{j}}
\end{pmatrix},
\end{align}
\begin{align}\label{march19equ427}
r\begin{pmatrix}
\widehat{B_{j}}\\ \widehat{E_{j}}\\\widehat{D_{j}}
\end{pmatrix}=r\begin{pmatrix}\widehat{B_{j}}\\\widehat{D_{j}}\end{pmatrix},
\end{align}
\begin{align}\label{march17equ426}
r\begin{pmatrix}
\widehat{A_{j}}&\widehat{E_{j}}\\
0&\widehat{D_{j}}
\end{pmatrix}=r(\widehat{A_{j}})+r(\widehat{D_{j}}),
\end{align}
\begin{align}\label{march19equ429}
r\begin{pmatrix}
\widehat{B_{j}}&0\\\widehat{E_{j}}&\widehat{C_{j}}
\end{pmatrix}=r(\widehat{B_{j}})+r(\widehat{C_{j}}),
\end{align}
\begin{align}\label{march17equ427}
&r\begin{pmatrix}
\widehat{A_{m}}&\widehat{E_{m}}&\widehat{C_{m}}\\
 &\widehat{D_{m}}& &\widehat{B_{m+1}}\\
 & & \widehat{A_{m+1}}&-\widehat{E_{m+1}}&\widehat{C_{m+1}}\\
  & & & \ddots & \ddots & \ddots\\
  & & & & \widehat{A_{n}}&(-1)^{n-m}\widehat{E_{n}}&\widehat{C_{n}}
\end{pmatrix}\nonumber\\&=r
\begin{pmatrix}
\widehat{A_{m}}&\widehat{C_{m}}\\
& \widehat{A_{m+1}}&\widehat{C_{m+1}}\\
& & \ddots&\ddots\\
& & & \widehat{A_{n}}&\widehat{C_{n}}
\end{pmatrix}+r
\begin{pmatrix}
\widehat{D_{m}}&\widehat{B_{m+1}}\\
& \widehat{D_{m+1}}&\widehat{B_{m+2}}\\
& & \ddots&\ddots\\
& & & \widehat{D_{n-1}}&\widehat{B_{n}}
\end{pmatrix},
\end{align}
\begin{align}\label{march17equ428}
&r
\begin{pmatrix}
\widehat{B_{m}}\\
\widehat{E_{m}}&\widehat{C_{m}}\\
\widehat{D_{m}}& &\widehat{B_{m+1}}\\
 &\widehat{A_{m+1}}&-\widehat{E_{m+1}}&\ddots\\
 & & \widehat{D_{m+1}}& \ddots & \widehat{B_{n}}\\
  & & &\ddots&(-1)^{n-m}\widehat{E_{n}}\\
  & & & & D_{n}
\end{pmatrix}\nonumber\\&=r
\begin{pmatrix}
\widehat{C_{m}}\\
 \widehat{A_{m+1}}&\widehat{C_{m+1}}\\
 &\widehat{A_{m+2}}&\ddots\\
 & &\ddots&\widehat{C_{n-1}}\\
 & & &\widehat{A_{n}}
\end{pmatrix}+r
\begin{pmatrix}
\widehat{B_{m}}\\
\widehat{D_{m}}&\widehat{B_{m+1}}\\
& \widehat{D_{m+1}}&\ddots\\
& & \ddots&\widehat{B_{n}}\\
& & & \widehat{D_{n}}
\end{pmatrix},
\end{align}
\begin{align}\label{march17equ429}
&r
\begin{pmatrix}
\widehat{A_{m}}&\widehat{E_{m}}&\widehat{C_{m}}\\
&\widehat{D_{m}}& &\widehat{B_{m+1}}\\
& &\widehat{A_{m+1}}&-\widehat{E_{m+1}}&\ddots\\
& & & \widehat{D_{m+1}}& \ddots & \widehat{B_{n}}\\
&  & & &\ddots&(-1)^{n-m}\widehat{E_{n}}\\
&  & & & & \widehat{D_{n}}
\end{pmatrix}\nonumber\\&=r
\begin{pmatrix}
\widehat{A_{m}}&\widehat{C_{m}}\\
& \widehat{A_{m+1}}&\widehat{C_{m+1}}\\
& & \ddots&\ddots\\
& & & \widehat{A_{n-1}}&\widehat{C_{n-1}}\\
& & & & \widehat{A_{n}}
\end{pmatrix}+r
\begin{pmatrix}
\widehat{D_{m}}&\widehat{B_{m+1}}\\
& \widehat{D_{m+1}}&\widehat{B_{m+2}}\\
& & \ddots&\ddots\\
& & & \widehat{D_{n-1}}&\widehat{B_{n}}\\
& & & & \widehat{D_{n}}
\end{pmatrix},
\end{align}
\begin{align}\label{march17equ430}
&r
\begin{pmatrix}
\widehat{B_{m}}\\
\widehat{E_{m}}&\widehat{C_{m}}\\
\widehat{D_{m}}& &\widehat{B_{m+1}}\\
& \widehat{A_{m+1}}&-\widehat{E_{m+1}}&\widehat{C_{m+1}}\\
 & & \ddots & \ddots & \ddots\\
 & & & \widehat{A_{n}}&(-1)^{n-m}\widehat{E_{n}}&\widehat{C_{n}}
\end{pmatrix}\nonumber\\&=r
\begin{pmatrix}
\widehat{C_{m}}\\
\widehat{A_{m+1}}&\widehat{C_{m+1}}\\
&\widehat{A_{m+2}}&\widehat{C_{m+2}}\\
&&\ddots &\ddots\\
& & & \widehat{A_{n}}&\widehat{C_{n}}
\end{pmatrix}+r
\begin{pmatrix}
\widehat{B_{m}}\\
\widehat{D_{m+1}}&\widehat{B_{m+1}}\\
&\widehat{D_{m+2}}&\widehat{B_{m+2}}\\
&&\ddots &\ddots\\
& & & \widehat{D_{n-1}}&\widehat{B_{n}}
\end{pmatrix}.
\end{align}Next we will prove that the rank equalities (\ref{march17equ425})-(\ref{march17equ430}) are equivalent with the rank equalities (\ref{march15equ034})-(\ref{march15equ037}). We establish some useful facts that will be used throughout this part.
\begin{description}
  \item[Fact 1] The expression of $F_{j}$ in (\ref{march17equfj}):  Since
  \begin{align}
  X^{2}_{j+1}:=A_{j+1}^{\dag}E_{j+1}B_{j+1}^{\dag}-A_{j+1}^{\dag}C_{j+1}M_{j+1}^{\dag}E_{j+1}B_{j+1}^{\dag}-A_{j+1}^{\dag}S_{j+1}%
C_{j+1}^{\dag}E_{j+1}N_{j+1}^{\dag}D_{j+1}B_{j+1}^{\dag}
  \end{align}and
  \begin{align}
  X^{1}_{j+1}:=M_{j}^{\dag}E_{j}D_{j}^{\dag}+S_{j}^{\dag}S_{j}C_{j}^{\dag}E_{j}N_{j}^{\dag}
  \end{align}are special solutions to equations
  \begin{align}
  A_{j+1}X_{j+1}B_{j+1}+C_{j+1}X_{j+2}D_{j+1}=E_{j+1}
  \end{align}and
  \begin{align}
  A_{j}X_{j}B_{j}+C_{j}X_{j+1}D_{j}=E_{j},
  \end{align}respectively, under the rank equalities (\ref{march15equ032}) and (\ref{march15equ033}). Hence,
  \begin{align}\label{march18equ438}
  F_{j}=X^{2}_{j+1}-X^{1}_{j+1},
  \end{align}where $X^{2}_{j+1}$ and $X^{1}_{j+1}$ satisfy the equations
  \begin{align}\label{march18equ439}
  A_{j+1}X^{2}_{j+1}B_{j+1}+C_{j+1}X^{1}_{j+2}D_{j+1}=E_{j+1}
  \end{align}and
  \begin{align}\label{march18equ440}
  A_{j}X^{2}_{j}B_{j}+C_{j}X^{1}_{j+1}D_{j}=E_{j}.
  \end{align}
  \item[Fact 2] Formulas about $S_{j+1}$: From
  \begin{align}
  S_{j+1}-A_{j+1}A_{j+1}^{\dag}S_{j+1}=R_{A_{j+1}}S_{j+1}=R_{A_{j+1}}C_{j+1}L_{M_{j+1}}=M_{j+1}L_{M_{j+1}}=0,
  \end{align}we infer that
  \begin{align}\label{march17ajsj}
  A_{j+1}A_{j+1}^{\dag}S_{j+1}=S_{j+1}.
  \end{align}
  \item[Fact 3] The ranks of $\begin{pmatrix}R_{N_{j}}\\R_{D_{j}}\end{pmatrix}$ and $R_{N_{j}}$: Applying Lemma \ref{hlemma03} to $r\begin{pmatrix}R_{N_{j}}\\R_{D_{j}}\end{pmatrix}-r(R_{N_{j}})$ gives
  \begin{align*}
  &r\begin{pmatrix}R_{N_{j}}\\R_{D_{j}}\end{pmatrix}-r(R_{N_{j}})\\=&
  r\begin{pmatrix}I&N_{j}&0\\I&0&D_{j}\end{pmatrix}-r\begin{pmatrix}I&N_{j}\end{pmatrix}-r(D_{j})\\=&
  r\begin{pmatrix}D_{j}&N_{j}\end{pmatrix}-r(D_{j})\xlongequal{N_{j}=D_{j}L_{B_{j}}}~0.
  \end{align*}Hence, we have
  \begin{align}\label{march19equ443}
  r\begin{pmatrix}R_{N_{j}}\\R_{D_{j}}\end{pmatrix}=r(R_{N_{j}}),
  \end{align}i.e.,
  \begin{align}\label{march19equ444}
  R_{D_{j}}=T_{j}R_{N_{j}},
  \end{align}where $T_{j}$ is a matrix.
  \item[Fact 4] Formulas about $R_{N_{j+1}}D_{j+1}B_{j+1}^{\dag}B_{j+1}:$ Note that
  \begin{align*}
  R_{N_{j+1}}D_{j+1}-R_{N_{j+1}}D_{j+1}B_{j+1}^{\dag}B_{j+1}=R_{N_{j+1}}D_{j+1}L_{B_{j+1}}=R_{N_{j+1}}N_{j+1}=0.
  \end{align*}Hence, we have
  \begin{align}\label{march19equ445}
  R_{N_{j+1}}D_{j+1}B_{j+1}^{\dag}B_{j+1}=R_{N_{j+1}}D_{j+1}.
  \end{align}
\end{description}

We show that  (\ref{march17equ425})-(\ref{march17equ430}) are equivalent with (\ref{march15equ034})-(\ref{march15equ037}) through the following three steps.

\textbf{Step 1.} We show that the rank equality (\ref{march17equ425}) is equivalent with (\ref{march15equ034}) for the case $n-m=1.$ It follows from Lemma \ref{hlemma03} that
\begin{align*}
r\begin{pmatrix}
\widehat{A_{j}}&\widehat{E_{j}}&\widehat{C_{j}}
\end{pmatrix}=r\begin{pmatrix}
\widehat{A_{j}}&\widehat{C_{j}}
\end{pmatrix}\Longleftrightarrow
\end{align*}\scriptsize
\begin{align*}
&r\begin{pmatrix}
R_{\begin{pmatrix}\begin{smallmatrix}L_{M_{j}}L_{S_{j}}&-L_{A_{j+1}}\end{smallmatrix}\end{pmatrix}}L_{M_{j}}&
R_{\begin{pmatrix}\begin{smallmatrix}L_{M_{j}}L_{S_{j}}&-L_{A_{j+1}}\end{smallmatrix}\end{pmatrix}}F_{j}
L_{\begin{pmatrix}\begin{smallmatrix}R_{D_{j}}\\-R_{B_{j+1}}\end{smallmatrix}\end{pmatrix}}&
R_{\begin{pmatrix}\begin{smallmatrix}L_{M_{j}}L_{S_{j}}&-L_{A_{j+1}}\end{smallmatrix}\end{pmatrix}}A_{j+1}^{\dag}S_{j+1}
\end{pmatrix}\\&=r\begin{pmatrix}
R_{\begin{pmatrix}\begin{smallmatrix}L_{M_{j}}L_{S_{j}}&-L_{A_{j+1}}\end{smallmatrix}\end{pmatrix}}L_{M_{j}}&
R_{\begin{pmatrix}\begin{smallmatrix}L_{M_{j}}L_{S_{j}}&-L_{A_{j+1}}\end{smallmatrix}\end{pmatrix}}A_{j+1}^{\dag}S_{j+1}
\end{pmatrix}
\end{align*}
\begin{align*}\normalsize
\xLeftrightarrow{\mbox{Add} ~\begin{pmatrix}\begin{smallmatrix}L_{M_{j}}L_{S_{j}}&-L_{A_{j+1}}\end{smallmatrix}\end{pmatrix},~M_{j},
\begin{pmatrix}\begin{smallmatrix}R_{D_{j}}\\-R_{B_{j+1}}\end{smallmatrix}\end{pmatrix}~\mbox{to both sides} }
\end{align*} \normalsize
\begin{align*}
&r\begin{pmatrix}
I&F_{j}&A_{j+1}^{\dag}S_{j+1}&L_{M_{j}}L_{S_{j}}&L_{A_{j+1}}\\
M_{j}&0&0&0&0\\
0&R_{D_{j}}&0&0&0\\
0&R_{B_{j+1}}&0&0&0
\end{pmatrix}\\=&r\begin{pmatrix}
I&A_{j+1}^{\dag}S_{j+1}&L_{M_{j}}L_{S_{j}}&L_{A_{j+1}}\\
M_{j}&0&0&0
\end{pmatrix}+r\begin{pmatrix}R_{D_{j}}\\R_{B_{j+1}}\end{pmatrix}
\xLeftrightarrow{\mbox{Add} ~A_{j+1} ~\mbox{to both sides} }
\end{align*}
\begin{align*}
&r\begin{pmatrix}
I&F_{j}&A_{j+1}^{\dag}S_{j+1}&L_{M_{j}}L_{S_{j}}&I\\
M_{j}&0&0&0&0\\
0&R_{D_{j}}&0&0&0\\
0&R_{B_{j+1}}&0&0&0\\
0&0&0&0&A_{j+1}
\end{pmatrix}\\=&r\begin{pmatrix}
I&A_{j+1}^{\dag}S_{j+1}&L_{M_{j}}L_{S_{j}}&I\\
M_{j}&0&0&0\\
0&0&0&A_{j+1}
\end{pmatrix}+r\begin{pmatrix}R_{D_{j}}\\R_{B_{j+1}}\end{pmatrix}\xLeftrightarrow{\mbox{Use}~ (\ref{march17ajsj})~\mbox{and~elementary~ operations}}
\end{align*}
\begin{align*}
&r\begin{pmatrix}
I&F_{j}&0&I\\
M_{j}&0&0&0\\
0&R_{D_{j}}&0&0\\
0&R_{B_{j+1}}&0&0\\
0&0&S_{j+1}&A_{j+1}
\end{pmatrix}\\=&r\begin{pmatrix}
I&0&I\\
M_{j}&0&0\\
0&S_{j+1}&A_{j+1}
\end{pmatrix}+r\begin{pmatrix}R_{D_{j}}\\R_{B_{j+1}}\end{pmatrix}
\xLeftrightarrow{\mbox{Add}~ A_{j},~D_{j},~B_{j+1},~M_{j+1} ~\mbox{to both sides}}
\end{align*}
\begin{align*}
&r\begin{pmatrix}
I&F_{j}&0&I&0&0&0\\
C_{j}&0&0&0&A_{j}&0&0\\
0&I&0&0&0&D_{j}&0\\
0&I&0&0&0&0&B_{j+1}\\
0&0&C_{j+1}&A_{j+1}&0&0&0\\
0&0&M_{j+1}&0&0&0&0&0
\end{pmatrix}\\=&r\begin{pmatrix}
I&0&I&0\\
C_{j}&0&0&A_{j}\\
0&C_{j+1}&A_{j+1}&0\\
0&M_{j+1}&0&0&0
\end{pmatrix}+r\begin{pmatrix}I&D_{j}&0\\I&0&B_{j+1}\end{pmatrix}
\end{align*}
\begin{align*}
\xLeftrightarrow{\mbox{Use}~M_{j+1}=R_{A_{j+1}}C_{j+1},~\mbox{and elementary operations}}
\end{align*}
\begin{align*}
&r\begin{pmatrix}
I&F_{j}&0&I&0&0&0\\
C_{j}&0&0&0&A_{j}&0&0\\
0&I&0&0&0&D_{j}&0\\
0&I&0&0&0&0&B_{j+1}\\
0&0&C_{j+1}&A_{j+1}&0&0&0
\end{pmatrix}\\=&r\begin{pmatrix}
I&0&I&0\\
C_{j}&0&0&A_{j}\\
0&C_{j+1}&A_{j+1}&0
\end{pmatrix}+r\begin{pmatrix}I&D_{j}&0\\I&0&B_{j+1}\end{pmatrix}
\xLeftrightarrow{\mbox{Use~}~(\ref{march18equ438})}
\end{align*}
\begin{align*}
&r\begin{pmatrix}
I&X_{j+1}^{2}-X_{j+1}^{1}&0&I&0&0&0\\
C_{j}&0&0&0&A_{j}&0&0\\
0&I&0&0&0&D_{j}&0\\
0&I&0&0&0&0&B_{j+1}\\
0&0&C_{j+1}&A_{j+1}&0&0&0
\end{pmatrix}=r\begin{pmatrix}
I&0&I&0\\
C_{j}&0&0&A_{j}\\
0&C_{j+1}&A_{j+1}&0
\end{pmatrix}\\&+r\begin{pmatrix}I&D_{j}&0\\I&0&B_{j+1}\end{pmatrix}
\xLeftrightarrow{\mbox{Use}~(\ref{march18equ439}), (\ref{march18equ440})~\mbox{and ~elementary~ operations}}
\end{align*}
\begin{align*}
&r\begin{pmatrix}
A_{j}&E_{j}&C_{j}\\
 &D_{j}& &B_{j+1}\\
 & & A_{j+1}&-E_{j+1}&C_{j+1}
\end{pmatrix}\\=&r
\begin{pmatrix}
A_{j}&C_{j}\\
& A_{j+1}&C_{j+1}
\end{pmatrix}+r
\begin{pmatrix}
D_{j}&B_{j+1}
\end{pmatrix}\xlongequal{\mbox{Put} ~m=j ~ \mbox{and} ~n=j+1~\mbox{in}~(\ref{march15equ034})}~ (\ref{march15equ034}).
\end{align*}
We have showed that the rank equality (\ref{march17equ425}) is equivalent with (\ref{march15equ034}) when $n-m=1.$

\textbf{Step 2.} Now we will prove that the rank equality (\ref{march19equ427}) is equivalent with (\ref{march15equ035}) when $n-m=1.$ Applying Lemma \ref{hlemma03} to (\ref{march19equ427}) yields
\begin{align*}
r\begin{pmatrix}
\widehat{B_{j}}\\ \widehat{E_{j}}\\\widehat{D_{j}}
\end{pmatrix}=r\begin{pmatrix}\widehat{B_{j}}\\\widehat{D_{j}}\end{pmatrix}
\Longleftrightarrow
\end{align*}
\begin{align*}
r\begin{pmatrix}
R_{N_{j}}L_{\begin{pmatrix}\begin{smallmatrix}R_{D_{j}}\\-R_{B_{j+1}}\end{smallmatrix}\end{pmatrix}}\\
R_{\begin{pmatrix}\begin{smallmatrix}L_{M_{j}}L_{S_{j}}&-L_{A_{j+1}}\end{smallmatrix}\end{pmatrix}}F_{j}
L_{\begin{pmatrix}\begin{smallmatrix}R_{D_{j}}\\-R_{B_{j+1}}\end{smallmatrix}\end{pmatrix}}\\
R_{N_{j+1}}D_{j+1}B_{j+1}^{\dag}L_{\begin{pmatrix}\begin{smallmatrix}R_{D_{j}}\\-R_{B_{j+1}}\end{smallmatrix}\end{pmatrix}}
\end{pmatrix}=r\begin{pmatrix}
R_{N_{j}}L_{\begin{pmatrix}\begin{smallmatrix}R_{D_{j}}\\-R_{B_{j+1}}\end{smallmatrix}\end{pmatrix}}\\
R_{N_{j+1}}D_{j+1}B_{j+1}^{\dag}L_{\begin{pmatrix}\begin{smallmatrix}R_{D_{j}}\\-R_{B_{j+1}}\end{smallmatrix}\end{pmatrix}}
\end{pmatrix}
\end{align*}
\begin{align*}
\xLeftrightarrow{\mbox{Add} ~\begin{pmatrix}\begin{smallmatrix}L_{M_{j}}L_{S_{j}}&-L_{A_{j+1}}\end{smallmatrix}\end{pmatrix},~
\begin{pmatrix}\begin{smallmatrix}R_{D_{j}}\\-R_{B_{j+1}}\end{smallmatrix}\end{pmatrix}~\mbox{to both sides} }
\end{align*}
\begin{align*}
r\begin{pmatrix}
F_{j}&L_{M_{j}}L_{S_{j}}&-L_{A_{j+1}}\\
R_{N_{j}}&0&0\\
R_{N_{j+1}}D_{j+1}B_{j+1}^{\dag}&0&0\\
R_{D_{j}}&0&0\\
R_{B_{j+1}}&0&0
\end{pmatrix}=
r\begin{pmatrix}
L_{M_{j}}L_{S_{j}}&-L_{A_{j+1}}
\end{pmatrix}+
r\begin{pmatrix}
R_{N_{j}}\\
R_{N_{j+1}}D_{j+1}B_{j+1}^{\dag}\\
R_{D_{j}}\\
R_{B_{j+1}}
\end{pmatrix}
\end{align*}
\begin{align*}
\xLeftrightarrow{\mbox{Use} ~(\ref{march19equ443}) ~\mbox{and} ~(\ref{march19equ444}) }
\end{align*}
\begin{align*}
r\begin{pmatrix}
F_{j}&L_{M_{j}}L_{S_{j}}&L_{A_{j+1}}\\
R_{N_{j}}&0&0\\
R_{N_{j+1}}D_{j+1}B_{j+1}^{\dag}&0&0\\
R_{B_{j+1}}&0&0
\end{pmatrix}=
r\begin{pmatrix}
L_{M_{j}}L_{S_{j}}&L_{A_{j+1}}
\end{pmatrix}+
r\begin{pmatrix}
R_{N_{j}}\\
R_{N_{j+1}}D_{j+1}B_{j+1}^{\dag}\\
R_{B_{j+1}}
\end{pmatrix}
\end{align*}
\begin{align*}
\xLeftrightarrow{\mbox{Add} ~A_{j+1},~B_{j+1},~S_{j},~N_{j}~\mbox{to both sides}}
\end{align*}
\begin{align*}
r\begin{pmatrix}
F_{j}&L_{M_{j}}&I&0&0\\
I&0&0&N_{j}&0\\
R_{N_{j+1}}D_{j+1}B_{j+1}^{\dag}&0&0&0&0\\
I&0&0&0&B_{j+1}\\
0&S_{j}&0&0&0\\
0&0&A_{j+1}&0&0
\end{pmatrix}=
r\begin{pmatrix}
L_{M_{j}}&I\\
S_{j}&0\\
0&A_{j+1}
\end{pmatrix}+
r\begin{pmatrix}
I&N_{j}&0\\
R_{N_{j+1}}D_{j+1}B_{j+1}^{\dag}&0&0\\
I&0&B_{j+1}
\end{pmatrix}
\end{align*}
\begin{align*}
\xLeftrightarrow{\mbox{Use}~(\ref{march19equ445})~\mbox{and elementary operations}}
\end{align*}
\begin{align*}
r\begin{pmatrix}
F_{j}&L_{M_{j}}&I&0&0\\
I&0&0&N_{j}&0\\
0&0&0&0&R_{N_{j+1}}D_{j+1}\\
I&0&0&0&B_{j+1}\\
0&C_{j}L_{M_{j}}&0&0&0\\
0&0&A_{j+1}&0&0
\end{pmatrix}=
r\begin{pmatrix}
L_{M_{j}}&I\\
C_{j}L_{M_{j}}&0\\
0&A_{j+1}
\end{pmatrix}+
r\begin{pmatrix}
I&N_{j}&0\\
0&0&R_{N_{j+1}}D_{j+1}\\
I&0&B_{j+1}
\end{pmatrix}
\end{align*}
\begin{align*}
\xLeftrightarrow{\mbox{Add}~M_{j}~\mbox{to both sides}}
\end{align*}
\begin{align*}
r\begin{pmatrix}
F_{j}&I&I&0&0\\
I&0&0&D_{j}L_{B_{j}}&0\\
0&0&0&0&R_{N_{j+1}}D_{j+1}\\
I&0&0&0&B_{j+1}\\
0&C_{j}&0&0&0\\
0&0&A_{j+1}&0&0
\end{pmatrix}=
r\begin{pmatrix}
I&I\\
C_{j}&0\\
0&A_{j+1}
\end{pmatrix}+
r\begin{pmatrix}
I&D_{j}L_{B_{j}}&0\\
0&0&R_{N_{j+1}}D_{j+1}\\
I&0&B_{j+1}
\end{pmatrix}
\end{align*}
\begin{align*}
\xLeftrightarrow{\mbox{Add}~N_{j+1}~\mbox{and}~B_{j}~\mbox{to both sides}}
\end{align*}
\begin{align*}
r\begin{pmatrix}
F_{j}&I&I&0&0\\
I&0&0&D_{j}&0\\
0&0&0&0&D_{j+1}\\
I&0&0&0&B_{j+1}\\
0&C_{j}&0&0&0\\
0&0&A_{j+1}&0&0\\
0&0&0&B_{j}&0
\end{pmatrix}=
r\begin{pmatrix}
I&I\\
C_{j}&0\\
0&A_{j+1}
\end{pmatrix}+
r\begin{pmatrix}
I&D_{j}&0\\
0&0&D_{j+1}\\
I&0&B_{j+1}\\
0&B_{j}&0
\end{pmatrix}
\xLeftrightarrow{\mbox{Use}~(\ref{march18equ438})}
\end{align*}
\begin{align*}
r\begin{pmatrix}
X_{j+1}^{2}-X_{j+1}^{1}&I&I&0&0\\
I&0&0&D_{j}&0\\
0&0&0&0&D_{j+1}\\
I&0&0&0&B_{j+1}\\
0&C_{j}&0&0&0\\
0&0&A_{j+1}&0&0\\
0&0&0&B_{j}&0
\end{pmatrix}=
r\begin{pmatrix}
I&I\\
C_{j}&0\\
0&A_{j+1}
\end{pmatrix}+
r\begin{pmatrix}
I&D_{j}&0\\
0&0&D_{j+1}\\
I&0&B_{j+1}\\
0&B_{j}&0
\end{pmatrix}
\end{align*}
\begin{align*}
\xLeftrightarrow{\mbox{Use elementary operations}}
\end{align*}
\begin{align*}
r
\begin{pmatrix}
B_{j}\\
E_{j}&C_{j}\\
D_{j}& &B_{j+1}\\
 &A_{j+1}&-E_{j+1}\\
 & & D_{j+1}
\end{pmatrix}=r
\begin{pmatrix}
C_{j}\\
 A_{j+1}
\end{pmatrix}+r
\begin{pmatrix}
B_{j}\\
D_{j}&B_{j+1}\\
& D_{j+1}
\end{pmatrix}\xlongequal{\mbox{Put} ~m=j ~ \mbox{and} ~n=j+1~\mbox{in}~(\ref{march15equ035})}~ (\ref{march15equ035}).
\end{align*}
Similarly, it can be found that
\begin{align*}
(\ref{march17equ426}) \Longleftrightarrow (\ref{march15equ036}),~ (m=j,~n=j+1),
\end{align*}
\begin{align*}
(\ref{march19equ429}) \Longleftrightarrow (\ref{march15equ037}),~ (m=j,~n=j+1).
\end{align*}

\textbf{Step 3.} We will prove that (\ref{march17equ427}) $\Longleftrightarrow$ (\ref{march15equ034}) for the case $n-m>1$. First, we only deal with $\widehat{A_{m}}, ~ \widehat{C_{m}},~\widehat{A_{m+1}}, ~\widehat{D_{m}},~\widehat{B_{m+1}},~ \widehat{E_{m}}$ in (\ref{march17equ427}). We want to find some rules. Applying Lemma \ref{hlemma03} to $\widehat{A_{m}}, ~ \widehat{C_{m}},~\widehat{A_{m+1}}, ~\widehat{D_{m}},~\widehat{B_{m+1}},~ \widehat{E_{m}}$ in (\ref{march17equ427}) that
\begin{align*}
&r\begin{pmatrix}
\widehat{A_{m}}&\widehat{E_{m}}&\widehat{C_{m}}\\
 &\widehat{D_{m}}& &\widehat{B_{m+1}}\\
 & & \widehat{A_{m+1}}&-\widehat{E_{m+1}}&\widehat{C_{m+1}}\\
  & & & \ddots & \ddots & \ddots\\
  & & & & \widehat{A_{n}}&(-1)^{n-m}\widehat{E_{n}}&\widehat{C_{n}}
\end{pmatrix}\\&=r
\begin{pmatrix}
\widehat{A_{m}}&\widehat{C_{m}}\\
& \widehat{A_{m+1}}&\widehat{C_{m+1}}\\
& & \ddots&\ddots\\
& & & \widehat{A_{n}}&\widehat{C_{n}}
\end{pmatrix}+r
\begin{pmatrix}
\widehat{D_{m}}&\widehat{B_{m+1}}\\
& \widehat{D_{m+1}}&\widehat{B_{m+2}}\\
& & \ddots&\ddots\\
& & & \widehat{D_{n-1}}&\widehat{B_{n}}
\end{pmatrix},
\end{align*}
\begin{align*}
\xLeftrightarrow{\mbox{Replace}~\widehat{A_{m}}, ~ \widehat{C_{m}},~\widehat{A_{m+1}}, ~\widehat{D_{m}},~\widehat{B_{m+1}},~ \widehat{E_{m}}~\mbox{by}~(\ref{may28equ421})-(\ref{may28equ423})}
\end{align*}
\tiny
\begin{align*}
&r\begin{pmatrix}\begin{smallmatrix}
R_{\begin{pmatrix}\begin{smallmatrix}L_{M_{m}}L_{S_{m}}&-L_{A_{m+1}}\end{smallmatrix}\end{pmatrix}}L_{M_{m}}&
R_{\begin{pmatrix}\begin{smallmatrix}L_{M_{m}}L_{S_{m}}&-L_{A_{m+1}}\end{smallmatrix}\end{pmatrix}}F_{m}
L_{\begin{pmatrix}\begin{smallmatrix}R_{D_{m}}\\-R_{B_{m+1}}\end{smallmatrix}\end{pmatrix}}&
R_{\begin{pmatrix}\begin{smallmatrix}L_{M_{m}}L_{S_{m}}&-L_{A_{m+1}}\end{smallmatrix}\end{pmatrix}}A_{m+1}^{\dag}S_{m+1}\\
 &R_{N_{m+1}}D_{m+1}B_{m+1}^{\dag}L_{\begin{pmatrix}\begin{smallmatrix}R_{D_{m}}\\-R_{B_{m+1}}\end{smallmatrix}\end{pmatrix}}& &
 R_{N_{m+1}}L_{\begin{pmatrix}\begin{smallmatrix}R_{D_{m+1}}\\-R_{B_{m+2}}\end{smallmatrix}\end{pmatrix}}\\
 & & R_{\begin{pmatrix}\begin{smallmatrix}L_{M_{m+1}}L_{S_{m+1}}&-L_{A_{m+2}}\end{smallmatrix}\end{pmatrix}}L_{M_{m+1}}&-\widehat{E_{m+1}}&\widehat{C_{m+1}}\\
 & & \ddots&\ddots&\ddots\end{smallmatrix}
\end{pmatrix}\\&
=r
\begin{pmatrix}
R_{\begin{pmatrix}\begin{smallmatrix}L_{M_{m}}L_{S_{m}}&-L_{A_{m+1}}\end{smallmatrix}\end{pmatrix}}L_{M_{m}}&
R_{\begin{pmatrix}\begin{smallmatrix}L_{M_{m}}L_{S_{m}}&-L_{A_{m+1}}\end{smallmatrix}\end{pmatrix}}A_{m+1}^{\dag}S_{m+1}\\
& R_{\begin{pmatrix}\begin{smallmatrix}L_{M_{m+1}}L_{S_{m+1}}&-L_{A_{m+2}}\end{smallmatrix}\end{pmatrix}}L_{M_{m+1}}&\widehat{C_{m+1}}\\
& & \ddots&\ddots
\end{pmatrix}\\&+r
\begin{pmatrix}
R_{N_{m+1}}D_{m+1}B_{m+1}^{\dag}L_{\begin{pmatrix}\begin{smallmatrix}R_{D_{m}}\\-R_{B_{m+1}}\end{smallmatrix}\end{pmatrix}}&
 R_{N_{m+1}}L_{\begin{pmatrix}\begin{smallmatrix}R_{D_{m+1}}\\-R_{B_{m+2}}\end{smallmatrix}\end{pmatrix}}\\
& \widehat{D_{m+1}}&\widehat{B_{m+2}}\\
& & \ddots&\ddots
\end{pmatrix},
\end{align*}\normalsize
\begin{align*}
\xLeftrightarrow{\mbox{Add}~M_{m},~\begin{pmatrix}\begin{smallmatrix}L_{M_{m}}L_{S_{m}}&-L_{A_{m+1}}\end{smallmatrix}\end{pmatrix}~\mbox{and}~ \begin{pmatrix}\begin{smallmatrix}R_{D_{m}}\\-R_{B_{m+1}}\end{smallmatrix}\end{pmatrix}~\mbox{to both sides}}
\end{align*}
\begin{align*}
&r\begin{pmatrix}\begin{smallmatrix}
I&F_{m}&A_{m+1}^{\dag}S_{m+1}&L_{M_{m}}L_{S_{m}}&L_{A_{m+1}}&0&0\\
0&R_{N_{m+1}}D_{m+1}B_{m+1}^{\dag}&0&0&0&R_{N_{m+1}}L_{\begin{pmatrix}\begin{smallmatrix}R_{D_{m+1}}\\-R_{B_{m+2}}\end{smallmatrix}\end{pmatrix}}&0\\
0&R_{D_{m}}&0&0&0&0&0\\
0&R_{B_{m+1}}&0&0&0&0&0\\
M_{m}&0&0&0&0&0&0\\
0&0&R_{\begin{pmatrix}\begin{smallmatrix}L_{M_{m+1}}L_{S_{m+1}}&-L_{A_{m+2}}\end{smallmatrix}\end{pmatrix}}L_{M_{m+1}}&0&0&-\widehat{E_{m+1}}&\widehat{C_{m+1}}\\
&&&&&\ddots&\ddots
\end{smallmatrix}
\end{pmatrix}\\
&=r\begin{pmatrix}\begin{smallmatrix}
I&A_{m+1}^{\dag}S_{m+1}&L_{M_{m}}L_{S_{m}}&L_{A_{m+1}}&0\\
M_{m}&0&0&0&0\\
0&R_{\begin{pmatrix}\begin{smallmatrix}L_{M_{m+1}}L_{S_{m+1}}&-L_{A_{m+2}}\end{smallmatrix}\end{pmatrix}}L_{M_{m+1}}&0&0&\widehat{C_{m+1}}\\
&& & \ddots&\ddots
\end{smallmatrix}\end{pmatrix}\\&+r\begin{pmatrix}\begin{smallmatrix}
R_{N_{m+1}}D_{m+1}B_{m+1}^{\dag}&R_{N_{m+1}}L_{\begin{pmatrix}\begin{smallmatrix}R_{D_{m+1}}\\-R_{B_{m+2}}\end{smallmatrix}\end{pmatrix}}\\
R_{D_{m}}&0\\
R_{B_{m+1}}&0\\
&\ddots&\ddots
\end{smallmatrix}
\end{pmatrix}
\end{align*}
\begin{align*}
\xLeftrightarrow{\mbox{Use~ elementary ~operations}}
\end{align*}
\begin{align*}
&r\begin{pmatrix}
I&F_{m}&A_{m+1}^{\dag}S_{m+1}&L_{A_{m+1}}&0&0\\
0&R_{N_{m+1}}D_{m+1}B_{m+1}^{\dag}&0&0&R_{N_{m+1}}L_{\begin{pmatrix}\begin{smallmatrix}R_{D_{m+1}}\\-R_{B_{m+2}}\end{smallmatrix}\end{pmatrix}}&0\\
0&R_{D_{m}}&0&0&0&0\\
0&R_{B_{m+1}}&0&0&0&0\\
M_{m}&0&0&0&0&0\\
0&0&R_{\begin{pmatrix}\begin{smallmatrix}L_{M_{m+1}}L_{S_{m+1}}&-L_{A_{m+2}}\end{smallmatrix}\end{pmatrix}}L_{M_{m+1}}&0&-\widehat{E_{m+1}}&\widehat{C_{m+1}}\\
&&&&&\ddots&\ddots
\end{pmatrix}\\&
=r\begin{pmatrix}
I&A_{m+1}^{\dag}S_{m+1}&L_{A_{m+1}}&0\\
M_{m}&0&0&0\\
0&R_{\begin{pmatrix}\begin{smallmatrix}L_{M_{m+1}}L_{S_{m+1}}&-L_{A_{m+2}}\end{smallmatrix}\end{pmatrix}}L_{M_{m+1}}&0&\widehat{C_{m+1}}\\
&&&&\ddots&\ddots
\end{pmatrix}\\&+
r\begin{pmatrix}
R_{N_{m+1}}D_{m+1}B_{m+1}^{\dag}&R_{N_{m+1}}L_{\begin{pmatrix}\begin{smallmatrix}R_{D_{m+1}}\\-R_{B_{m+2}}\end{smallmatrix}\end{pmatrix}}&0\\
R_{D_{m}}&0&0\\
R_{B_{m+1}}&0&0\\
&&&\ddots&\ddots
\end{pmatrix}
\end{align*}
\begin{align*}
\xLeftrightarrow{\mbox{Add}~A_{m+1},~A_{m},~D_{m},~\mbox{and}~B_{m+1}~\mbox{to both sides}}
\end{align*}\tiny
\begin{align*}
&r\begin{pmatrix}
I&F_{m}&A_{m+1}^{\dag}S_{m+1}&I&0&0&0&0&0\\
0&R_{N_{m+1}}D_{m+1}B_{m+1}^{\dag}&0&0&0&0&0&R_{N_{m+1}}L_{\begin{pmatrix}\begin{smallmatrix}R_{D_{m+1}}\\-R_{B_{m+2}}\end{smallmatrix}\end{pmatrix}}&0\\
0&I&0&0&D_{m}&0&0&0&0\\
0&I&0&0&0&B_{m+1}&0&0&0\\
C_{m}&0&0&0&0&0&A_{m}&0&0\\
0&0&0&A_{m+1}&0&0&0&0&0\\
0&0&R_{\begin{pmatrix}\begin{smallmatrix}L_{M_{m+1}}L_{S_{m+1}}&-L_{A_{m+2}}\end{smallmatrix}\end{pmatrix}}L_{M_{m+1}}&0&0&0&0&-\widehat{E_{m+1}}&\widehat{C_{m+1}}\\
&&&&&&&\ddots&\ddots
\end{pmatrix}\\
&=r\begin{pmatrix}
I&A_{m+1}^{\dag}S_{m+1}&I&0&0\\
C_{m}&0&0&A_{m}&0\\
0&0&A_{m+1}&0&0\\
0&R_{\begin{pmatrix}\begin{smallmatrix}L_{M_{m+1}}L_{S_{m+1}}&-L_{A_{m+2}}\end{smallmatrix}\end{pmatrix}}L_{M_{m+1}}&0&0&\widehat{C_{m+1}}
\\
&&&&\ddots&\ddots
\end{pmatrix}\\&+r\begin{pmatrix}
R_{N_{m+1}}D_{m+1}B_{m+1}^{\dag}&0&0&0&R_{N_{m+1}}L_{\begin{pmatrix}\begin{smallmatrix}R_{D_{m+1}}\\-R_{B_{m+2}}\end{smallmatrix}\end{pmatrix}}\\
I&D_{m}&0&0&0\\
I&0&B_{m+1}&0&0\\
&&&\ddots&\ddots
\end{pmatrix}
\end{align*}
\normalsize
\begin{align*}
\xLeftrightarrow{\mbox{Use}~(\ref{march17ajsj})~\mbox{and}~(\ref{march19equ445})}
\end{align*}\tiny
\begin{align*}
&r\begin{pmatrix}
I&F_{m}&0&I&0&0&0&0\\
0&0&0&0&0&-R_{N_{m+1}}D_{m+1}&0&R_{N_{m+1}}L_{\begin{pmatrix}\begin{smallmatrix}R_{D_{m+1}}\\-R_{B_{m+2}}\end{smallmatrix}\end{pmatrix}}\\
0&I&0&0&D_{m}&0&0&0\\
0&I&0&0&0&B_{m+1}&0&0\\
C_{m}&0&0&0&0&0&A_{m}&0\\
0&0&-S_{m+1}&A_{m+1}&0&0&0&0\\
0&0&R_{\begin{pmatrix}\begin{smallmatrix}L_{M_{m+1}}L_{S_{m+1}}&-L_{A_{m+2}}\end{smallmatrix}\end{pmatrix}}L_{M_{m+1}}&0&0&0&0&-\widehat{E_{m+1}}&\widehat{C_{m+1}}\\
&&&&&&&\ddots&\ddots
\end{pmatrix}\\&=
r\begin{pmatrix}
I&0&I&0&0\\
C_{m}&0&0&A_{m}&0\\
0&-S_{m+1}&A_{m+1}&0&0\\
0&R_{\begin{pmatrix}\begin{smallmatrix}L_{M_{m+1}}L_{S_{m+1}}&-L_{A_{m+2}}\end{smallmatrix}\end{pmatrix}}L_{M_{m+1}}&0&0&\widehat{C_{m+1}}\\
&&&&\ddots&\ddots
\end{pmatrix}\\&+r\begin{pmatrix}
0&0&-R_{N_{m+1}}D_{m+1}&R_{N_{m+1}}L_{\begin{pmatrix}\begin{smallmatrix}R_{D_{m+1}}\\-R_{B_{m+2}}\end{smallmatrix}\end{pmatrix}}\\
I&D_{m}&0&0\\
I&0&B_{m+1}&0\\
&&\ddots&\ddots
\end{pmatrix}
\end{align*}
\normalsize
\begin{align*}
\xLeftrightarrow{\mbox{Add}~N_{m+1}~\mbox{and}~M_{m+1} ~\mbox{to both sides}}
\end{align*}\tiny
\begin{align*}
&r\begin{pmatrix}
I&F_{m}&0&I&0&0&0&0&0&0\\
0&0&0&0&0&-D_{m+1}&N_{m+1}&0&L_{\begin{pmatrix}\begin{smallmatrix}R_{D_{m+1}}\\-R_{B_{m+2}}\end{smallmatrix}\end{pmatrix}}&0\\
0&I&0&0&D_{m}&0&0&0&0&0\\
0&I&0&0&0&B_{m+1}&0&0&0&0\\
C_{m}&0&0&0&0&0&0&A_{m}&0&0\\
0&0&-C_{m+1}&A_{m+1}&0&0&0&0&0&0\\
0&0&M_{m+1}&0&0&0&0&0&0&0\\
0&0&R_{\begin{pmatrix}\begin{smallmatrix}L_{M_{m+1}}L_{S_{m+1}}&-L_{A_{m+2}}\end{smallmatrix}\end{pmatrix}}&0&0&0&0&0&-\widehat{E_{m+1}}&\widehat{C_{m+1}}\\
&&&&&&&\ddots&\ddots
\end{pmatrix}\\&=
r\begin{pmatrix}
I&0&I&0&0\\
C_{m}&0&0&A_{m}&0\\
0&-C_{m+1}&A_{m+1}&0&0\\
0&M_{m+1}&0&0&0\\
0&R_{\begin{pmatrix}\begin{smallmatrix}L_{M_{m+1}}L_{S_{m+1}}&-L_{A_{m+2}}\end{smallmatrix}\end{pmatrix}}&0&0&\widehat{C_{m+1}}\\
&&&\ddots&\ddots
\end{pmatrix}\\&+
r\begin{pmatrix}
0&0&-D_{m+1}&N_{m+1}&L_{\begin{pmatrix}\begin{smallmatrix}R_{D_{m+1}}\\-R_{B_{m+2}}\end{smallmatrix}\end{pmatrix}}\\
I&D_{m}&0&0&0\\
I&0&B_{m+1}&0&0\\
&&&\ddots&\ddots
\end{pmatrix}
\end{align*}
\normalsize
\begin{align*}
\xLeftrightarrow{\mbox{Use~ elementary ~operations}}
\end{align*}\footnotesize
\begin{align*}
&r\begin{pmatrix}
I&F_{m}&0&I&0&0&0&0&0\\
0&0&0&0&0&-D_{m+1}&0&L_{\begin{pmatrix}\begin{smallmatrix}R_{D_{m+1}}\\-R_{B_{m+2}}\end{smallmatrix}\end{pmatrix}}&0\\
0&I&0&0&D_{m}&0&0&0&0\\
0&I&0&0&0&B_{m+1}&0&0&0\\
C_{m}&0&0&0&0&0&A_{m}&0&0\\
0&0&-C_{m+1}&A_{m+1}&0&0&0&0&0\\
0&0&R_{\begin{pmatrix}\begin{smallmatrix}L_{M_{m+1}}L_{S_{m+1}}&-L_{A_{m+2}}\end{smallmatrix}\end{pmatrix}}&0&0&0&0&-\widehat{E_{m+1}}&\widehat{C_{m+1}}\\
&&&&&&&&\ddots&\ddots
\end{pmatrix}\\&=
r\begin{pmatrix}
I&0&I&0&0\\
C_{m}&0&0&A_{m}&0\\
0&-C_{m+1}&A_{m+1}&0&0\\
0&R_{\begin{pmatrix}\begin{smallmatrix}L_{M_{m+1}}L_{S_{m+1}}&-L_{A_{m+2}}\end{smallmatrix}\end{pmatrix}}&0&0&\widehat{C_{m+1}}\\
&&&\ddots&\ddots
\end{pmatrix}\\&+
r\begin{pmatrix}
0&0&-D_{m+1}&L_{\begin{pmatrix}\begin{smallmatrix}R_{D_{m+1}}\\-R_{B_{m+2}}\end{smallmatrix}\end{pmatrix}}\\
I&D_{m}&0&0\\
I&0&B_{m+1}&0\\
&&\ddots&\ddots
\end{pmatrix}
\end{align*}
\normalsize
\begin{align*}
\xLeftrightarrow{\mbox{Use}~(\ref{march18equ438})}
\end{align*}\footnotesize
\begin{align*}
&r\begin{pmatrix}
I&X^{2}_{m+1}-X^{1}_{m+1}&0&I&0&0&0&0&0\\
0&0&0&0&0&-D_{m+1}&0&L_{\begin{pmatrix}\begin{smallmatrix}R_{D_{m+1}}\\-R_{B_{m+2}}\end{smallmatrix}\end{pmatrix}}&0\\
0&I&0&0&D_{m}&0&0&0&0\\
0&I&0&0&0&B_{m+1}&0&0&0\\
C_{m}&0&0&0&0&0&A_{m}&0&0\\
0&0&-C_{m+1}&A_{m+1}&0&0&0&0&0\\
0&0&R_{\begin{pmatrix}\begin{smallmatrix}L_{M_{m+1}}L_{S_{m+1}}&-L_{A_{m+2}}\end{smallmatrix}\end{pmatrix}}&0&0&0&0&-\widehat{E_{m+1}}&\widehat{C_{m+1}}\\
&&&&&&&&\ddots&\ddots
\end{pmatrix}\\&=
r\begin{pmatrix}
I&0&I&0&0\\
C_{m}&0&0&A_{m}&0\\
0&-C_{m+1}&A_{m+1}&0&0\\
0&R_{\begin{pmatrix}\begin{smallmatrix}L_{M_{m+1}}L_{S_{m+1}}&-L_{A_{m+2}}\end{smallmatrix}\end{pmatrix}}&0&0&\widehat{C_{m+1}}\\
&&&\ddots&\ddots
\end{pmatrix}\\&+r\begin{pmatrix}
0&0&-D_{m+1}&L_{\begin{pmatrix}\begin{smallmatrix}R_{D_{m+1}}\\-R_{B_{m+2}}\end{smallmatrix}\end{pmatrix}}\\
I&D_{m}&0&0\\
I&0&B_{m+1}&0\\
&&\ddots&\ddots
\end{pmatrix}
\end{align*}\normalsize
\begin{align*}
\xLeftrightarrow{\mbox{Use~ elementary ~operations,~(\ref{march18equ439})~and ~(\ref{march18equ440})}}
\end{align*}\footnotesize
\begin{align*}
&r\begin{pmatrix}
I&0&0&0&0&0&0&0&0\\
0&0&0&0&0&-D_{m+1}&0&L_{\begin{pmatrix}\begin{smallmatrix}R_{D_{m+1}}\\-R_{B_{m+2}}\end{smallmatrix}\end{pmatrix}}&0\\
0&I&0&0&0&0&0&0&0\\
0&0&0&0&-D_{m}&B_{m+1}&0&0&0\\
0&0&0&-C_{m}&-E_{m}&0&A_{m}&0&0\\
0&0&-C_{m+1}&A_{m+1}&0&A_{m+1}X^{2}_{m+1}B_{m+1}&0&0&0\\
0&0&R_{\begin{pmatrix}\begin{smallmatrix}L_{M_{m+1}}L_{S_{m+1}}&-L_{A_{m+2}}\end{smallmatrix}\end{pmatrix}}&0&0&0&0&-\widehat{E_{m+1}}&\widehat{C_{m+1}}\\
&&&&&&&&\ddots&\ddots
\end{pmatrix}\\&=
r\begin{pmatrix}
I&0&0&0\\
0&0&-C_{m}&A_{m}\\
0&-C_{m+1}&A_{m+1}&0\\
0&R_{\begin{pmatrix}\begin{smallmatrix}L_{M_{m+1}}L_{S_{m+1}}&-L_{A_{m+2}}\end{smallmatrix}\end{pmatrix}}&0&0\\
&&\ddots&\ddots
\end{pmatrix}+r\begin{pmatrix}
0&0&-D_{m+1}&L_{\begin{pmatrix}\begin{smallmatrix}R_{D_{m+1}}\\-R_{B_{m+2}}\end{smallmatrix}\end{pmatrix}}\\
I&0&0&0\\
0&-D_{m}&B_{m+1}&0\\
&&\ddots&\ddots
\end{pmatrix}
\end{align*}\normalsize
\begin{align*}
\Longleftrightarrow
\end{align*}
\begin{align}\label{may28equ446}
&r\begin{pmatrix}
A_{m}&E_{m}&C_{m}\\
&D_{m}& & B_{m+1}\\
&&A_{m+1}&-A_{m+1}X_{m+1}^{2}B_{m+1}&C_{m+1}\\
&&&D_{m+1}&&-L_{\begin{pmatrix}\begin{smallmatrix}R_{D_{m+1}}\\-R_{B_{m+2}}\end{smallmatrix}\end{pmatrix}}\\
&&&&R_{\begin{pmatrix}\begin{smallmatrix}L_{M_{m+1}}L_{S_{m+1}}&-L_{A_{m+2}}\end{smallmatrix}\end{pmatrix}}&-\widehat{E_{m+1}}&\widehat{C_{m+1}}\\
&&&&&\widehat{D_{m+1}}&&\widehat{B_{m+2}}\\
&&&&&&\widehat{A_{m+2}}\\
&&&&&&&\ddots&\ddots
\end{pmatrix}\nonumber\\&=
r\begin{pmatrix}
A_{m}&C_{m}\\
&A_{m+1}&C_{m+1}\\
&&R_{\begin{pmatrix}\begin{smallmatrix}L_{M_{m+1}}L_{S_{m+1}}&-L_{A_{m+2}}\end{smallmatrix}\end{pmatrix}}&\widehat{C_{m+1}}\\
&&&\widehat{A_{m+2}}&\ddots
\end{pmatrix}\nonumber\\&+r\begin{pmatrix}
D_{m}& & B_{m+1}\\
&&D_{m+1}&&-L_{\begin{pmatrix}\begin{smallmatrix}R_{D_{m+1}}\\-R_{B_{m+2}}\end{smallmatrix}\end{pmatrix}}\\
&&&&\widehat{D_{m+1}}&\widehat{B_{m+2}}\\
&&&&&\ddots&\ddots
\end{pmatrix}
\end{align}\normalsize
Continuing in this way, we obtain that
\begin{align*}
(\ref{may28equ446})\Longleftrightarrow
\end{align*}\footnotesize
\begin{align*}
&r\begin{pmatrix}
A_{m}&E_{m}&C_{m}\\
&D_{m}& & B_{m+1}\\
&&A_{m+1}&-E_{m+1}&C_{m+1}\\
&&&\ddots&\ddots&\ddots\\
& & & &D_{n-1}& &B_{n}\\
& & & &&A_{n} & (-1)^{n-m}A_{n}X_{n}^{2}B_{n}& C_{n}\\
& & & &&& D_{n} & & -L_{\begin{pmatrix}\begin{smallmatrix}R_{D_{n}}\\-R_{B_{n+1}}\end{smallmatrix}\end{pmatrix}}\\
&&&&&&&R_{\begin{pmatrix}\begin{smallmatrix}L_{M_{n}}L_{S_{n}}&-L_{A_{n+1}}\end{smallmatrix}\end{pmatrix}}&(-1)^{n-m}\widehat{E_{n}}&\widehat{C_{n}}
\end{pmatrix}\\&=\begin{pmatrix}
A_{m}&C_{m}\\
& A_{m+1}&C_{m+1}\\
& & \ddots&\ddots\\
& & & A_{n}&C_{n}\\
&&&&R_{\begin{pmatrix}\begin{smallmatrix}L_{M_{n}}L_{S_{n}}&-L_{A_{n+1}}\end{smallmatrix}\end{pmatrix}}&\widehat{C_{n}}
\end{pmatrix}+r
\begin{pmatrix}
D_{m}&B_{m+1}\\
& D_{m+1}&B_{m+2}\\
& & \ddots&\ddots\\
& & & D_{n}&L_{\begin{pmatrix}\begin{smallmatrix}R_{D_{n}}\\-R_{B_{n+1}}\end{smallmatrix}\end{pmatrix}}
\end{pmatrix}
\end{align*}
\normalsize
\begin{align*}
\xLeftrightarrow{\mbox{Add}~\begin{pmatrix}\begin{smallmatrix}L_{M_{n}}L_{S_{n}}&-L_{A_{n+1}}\end{smallmatrix}\end{pmatrix}~\mbox{and}~
\begin{pmatrix}\begin{smallmatrix}R_{D_{n}}\\-R_{B_{n+1}}\end{smallmatrix}\end{pmatrix} ~\mbox{to both sides}}
\end{align*}\footnotesize
\begin{align*}
&r\begin{pmatrix}
A_{m}&E_{m}&C_{m}\\
&D_{m}& & B_{m+1}\\
&&A_{m+1}&-E_{m+1}&C_{m+1}\\
&&&\ddots&\ddots&\ddots\\
& & & &D_{n-1}& &B_{n}\\
& & & &&A_{n} & (-1)^{n-m}A_{n}X_{n}^{2}B_{n}& C_{n}\\
& & & &&& D_{n} & & I\\
&&&&&&&I&(-1)^{n-m+1}F_{n}&A^{\dag}_{n+1}S_{n+1}&L_{A_{n+1}}\\
&&&&&&&&R_{D_{n}}&\\
&&&&&&&&R_{B_{n+1}}&
\end{pmatrix}\\&=r\begin{pmatrix}
A_{m}&C_{m}\\
& A_{m+1}&C_{m+1}\\
& & \ddots&\ddots\\
& & & A_{n}&C_{n}\\
&&&&I& A^{\dag}_{n+1}S_{n+1}&L_{A_{n+1}}
\end{pmatrix}+r
\begin{pmatrix}
D_{m}&B_{m+1}\\
& D_{m+1}&B_{m+2}\\
& & \ddots&\ddots\\
& & & D_{n}&I\\
&&&&R_{D_{n}}\\
&&&&R_{B_{n+1}}
\end{pmatrix}
\end{align*}
\normalsize
\begin{align*}
\xLeftrightarrow{\mbox{Add}~D_{n}~\mbox{and}~B_{n+1} ~\mbox{to both sides and use (\ref{march17ajsj}), (\ref{march18equ439}) and (\ref{march18equ440})}}
\end{align*}\footnotesize
\begin{align*}
&r\begin{pmatrix}
A_{m}&E_{m}&C_{m}\\
&\ddots&\ddots&\ddots\\
 & &D_{n-1}& &B_{n}\\
 & &&A_{n} & (-1)^{n-m}A_{n}X_{n}^{2}B_{n}& C_{n}\\
 & &&& D_{n} & & I\\
&&&&&I&(-1)^{n-m+1}F_{n}&&I\\
&&&&&&I&&&D_{n}\\
&&&&&&I&&&&B_{n+1}\\
&&&&&&&C_{n+1}&A_{n+1}\\
\end{pmatrix}\\&=r\begin{pmatrix}
A_{m}&C_{m}\\
& A_{m+1}&C_{m+1}\\
& & \ddots&\ddots\\
& & & A_{n}&C_{n}\\
&&&&I& &I\\
&&&&& C_{n+1}&A_{n+1}
\end{pmatrix}+r
\begin{pmatrix}
D_{m}&B_{m+1}\\
& D_{m+1}&B_{m+2}\\
& & \ddots&\ddots\\
& & & D_{n}&I\\
&&&&I&D_{n}\\
&&&&&&B_{n+1}
\end{pmatrix}
\end{align*}\normalsize
\begin{align*}
\xLeftrightarrow{\mbox{Use~ elementary ~operations}}
\end{align*}
\begin{align*}
&r\begin{pmatrix}
A_{m}&E_{m}&C_{m}\\
&\ddots&\ddots&\ddots\\
 & &D_{n-1}& &B_{n}\\
 & &&A_{n} & (-1)^{n-m}E_{n}& C_{n}\\
 & &&& D_{n} & & B_{n+1}\\
&&&&&A_{n+1}&(-1)^{n-m+1}E_{n+1}&C_{n+1}
\end{pmatrix}\\&=r\begin{pmatrix}
A_{m}&C_{m}\\
& A_{m+1}&C_{m+1}\\
& & \ddots&\ddots\\
& & & A_{n+1}&C_{n+1}
\end{pmatrix}+r
\begin{pmatrix}
D_{m}&B_{m+1}\\
& D_{m+1}&B_{m+2}\\
& & \ddots&\ddots\\
& & & D_{n}&B_{n+1}
\end{pmatrix}
\end{align*}
\begin{align*}
\Longleftrightarrow (\ref{march15equ034}),~(n-m>1).
\end{align*}
Similarly, it can be found that
\begin{align*}
(\ref{march17equ428}) \Longleftrightarrow (\ref{march15equ035}),~(n-m>1),
\end{align*}
\begin{align*}
(\ref{march17equ429}) \Longleftrightarrow (\ref{march15equ036}),~(n-m>1),
\end{align*}
\begin{align*}
(\ref{march17equ430}) \Longleftrightarrow (\ref{march15equ037}),~(n-m>1).
\end{align*}

As special cases of Theorem \ref{maintheorem}, solvability conditions to the following systems of one-sided Sylvester-type quaternion matrix equations can be given
\begin{align}
\begin{cases}
A_{1}X_{1}+X_{2}D_{1}=E_{1},\\
A_{2}X_{2}+X_{3}D_{2}=E_{2},\\
A_{3}X_{3}+X_{4}D_{3}=E_{3},\\
\quad \quad \vdots\\
A_{k}X_{k}+X_{k+1}D_{k}=E_{k},
\end{cases}
\end{align}
\begin{align}
\begin{cases}
A_{1}X_{1}+X_{2}D_{1}=E_{1},\\
A_{2}X_{3}+X_{2}D_{2}=E_{2},\\
A_{3}X_{3}+X_{4}D_{3}=E_{3},\\
\quad \quad \vdots\\
A_{k}X_{2k+1}+X_{2k}D_{k}=E_{2k},
\end{cases}
\end{align}
\begin{align}
A_{i}X_{k}-X_{j}B_{i}=C_{i},~i=1,\ldots,n,~k\neq j,~k,j\in \{i,i+1\}.
\end{align}
Some authors have considered the solvability conditions to one-sided Sylvester-type matrix equations (e.g., \cite{Dmytryshyn}, \cite{Flanders}, \cite{Roth}, \cite{H.K.Wimmer}).

\section{\textbf{Solvability conditions to the system (\ref{applicationsystem})}}
In this section, we use Theorem \ref{maintheorem} to give some solvability conditions to the system of quaternion matrix equations involving $\eta$-Hermicity
\begin{align}\label{may29equ061}
\begin{cases}
&A_{1}X_{1}A^{\eta*}_{1}+C_{1}X_{2}C^{\eta*}_{1}=E_{1},\\
&\qquad \qquad \quad A_{2}X_{2}A^{\eta*}_{2}+C_{2}X_{3}C^{\eta*}_{2}=E_{2},\\
&\qquad \qquad \quad \qquad \qquad \quad A_{3}X_{3}A^{\eta*}_{3}+C_{3}X_{4}C^{\eta*}_{3}=E_{3},\\
&\qquad \qquad \quad \qquad \qquad\qquad\qquad\qquad\ddots\\
&\qquad \qquad \quad \qquad \qquad\qquad\qquad\qquad A_{k}X_{k}A^{\eta*}_{k}+C_{k}X_{k+1}C^{\eta*}_{k}=E_{k},
\end{cases}~X_{i}=X^{\eta*}_{i}.
\end{align}
 At first, we give the definition of $\eta$-Hermitian quaternion matrix.
\begin{definition}[$\eta$-Hermitian Matrix]
\cite{Took4} For $\eta\in\{\mathbf{i},\mathbf{j},\mathbf{k}\}$, a quaternion matrix $A$ is said to be $\eta$-Hermitian if $A=A^{\eta*},$ where $A^{\eta*}=-\eta A^{*}\eta$.
\end{definition}

\begin{theorem}\label{theorem061}
The system (\ref{may29equ061}) has an $\eta$-Hermitian solution if and only if the following $k(k+1)$ rank equalities hold for all $i=1,\ldots,k$ and $1\leq m< n\leq k$
\begin{align}\label{may29equ062}
r\begin{pmatrix}
A_{i}&E_{i}&C_{i}
\end{pmatrix}=r\begin{pmatrix}
A_{i}&C_{i}
\end{pmatrix},
~
r\begin{pmatrix}
A_{i}&E_{i}\\
0&C^{\eta*}_{i}
\end{pmatrix}=r(A_{i})+r(C_{i}),
\end{align}
\begin{align}\label{may29equ063}
&r\begin{pmatrix}\begin{smallmatrix}
A_{m}&E_{m}&C_{m}\\
 &C^{\eta*}_{m}& &A^{\eta*}_{m+1}\\
 & & A_{m+1}&-E_{m+1}&C_{m+1}\\
  & & & \ddots & \ddots & \ddots\\
  & & & & A_{n}&(-1)^{n-m}E_{n}&C_{n}\end{smallmatrix}
\end{pmatrix}\nonumber\\&=r
\begin{pmatrix}\begin{smallmatrix}
A_{m}&C_{m}\\
& A_{m+1}&C_{m+1}\\
& & \ddots&\ddots\\
& & & A_{n}&C_{n}\end{smallmatrix}
\end{pmatrix}+r
\begin{pmatrix}\begin{smallmatrix}
C_{m}\\
 A_{m+1}&C_{m+1}\\
 &A_{m+2}&\ddots\\
 & &\ddots&C_{n-1}\\
 & & &A_{n}\end{smallmatrix}
\end{pmatrix},
\end{align}
\begin{align}\label{may29equ064}
&r
\begin{pmatrix}\begin{smallmatrix}
A_{m}&E_{m}&C_{m}\\
&C^{\eta*}_{m}& &A^{\eta*}_{m+1}\\
& &A_{m+1}&-E_{m+1}&\ddots\\
& & & C^{\eta*}_{m+1}& \ddots & A^{\eta*}_{n}\\
&  & & &\ddots&(-1)^{n-m}E_{n}\\
&  & & & & C^{\eta*}_{n}\end{smallmatrix}
\end{pmatrix}\nonumber\\&=r
\begin{pmatrix}\begin{smallmatrix}
A_{m}&C_{m}\\
& A_{m+1}&C_{m+1}\\
& & \ddots&\ddots\\
& & & A_{n-1}&C_{n-1}\\
& & & & A_{n}\end{smallmatrix}
\end{pmatrix}+r
\begin{pmatrix}\begin{smallmatrix}
C_{m}\\
A_{m+1}&C_{m+1}\\
&A_{m+2}&C_{m+2}\\
&&\ddots &\ddots\\
& & & A_{n}&C_{n}\end{smallmatrix}
\end{pmatrix},
\end{align}
where the blank entries in (\ref{may29equ062})-(\ref{may29equ064}) are all zeros.
\end{theorem}

\begin{proof}
The system (\ref{may29equ061}) has an $\eta$-Hermitian solution if and only if the following system is consistent
\begin{align*}
\begin{cases}
&A_{1}Y_{1}A^{\eta*}_{1}+C_{1}Y_{2}C^{\eta*}_{1}=E_{1},\\
&\qquad \qquad \quad A_{2}Y_{2}A^{\eta*}_{2}+C_{2}Y_{3}C^{\eta*}_{2}=E_{2},\\
&\qquad \qquad \quad \qquad \qquad \quad A_{3}Y_{3}A^{\eta*}_{3}+C_{3}Y_{4}C^{\eta*}_{3}=E_{3},\\
&\qquad \qquad \quad \qquad \qquad\qquad\qquad\qquad\ddots\\
&\qquad \qquad \quad \qquad \qquad\qquad\qquad\qquad A_{k}Y_{k}A^{\eta*}_{k}+C_{k}Y_{k+1}C^{\eta*}_{k}=E_{k}.
\end{cases}
\end{align*}In this case, the general $\eta$-Hermitian solution to the system (\ref{may29equ061}) can be expressed as
\begin{align*}
X_{i}=\frac{Y_{i}+Y^{\eta*}_{i}}{2}.
\end{align*} The solvability conditions (\ref{may29equ062})-(\ref{may29equ064}) can be obtained by using Theorem \ref{maintheorem}.
\end{proof}

\end{document}